\documentclass[11pt,a4paper]{amsart}

\usepackage{amsmath,amssymb,amsfonts,amsthm,xspace,bbm,amssymb}
\usepackage{graphicx,color}
\usepackage{geometry}
\usepackage[utf8]{inputenc}
\usepackage{amscd}
\usepackage{amsmath}
\usepackage{amssymb}
\usepackage{amsfonts}
\usepackage{amsthm}
\usepackage{bbm}
\usepackage{mathrsfs}
\usepackage{graphicx}
\usepackage{geometry}
\usepackage{bbm}
\usepackage{bm}
\usepackage{mathrsfs}
\usepackage{epsfig}

\usepackage{enumerate}

\usepackage[T1]{fontenc}
\usepackage{stmaryrd}
\usepackage{nccmath}
\usepackage[all]{xy}
\usepackage{blkarray}
\usepackage[normalem]{ulem}
\usepackage{multirow}
\usepackage{rotating}
\usepackage[usenames,dvipsnames]{xcolor}

\usepackage[colorlinks=true,linkcolor=NavyBlue,urlcolor=RoyalBlue,citecolor=PineGreen,bookmarks=true,bookmarksopen=true,bookmarksopenlevel=2,unicode=true,linktocpage]{hyperref}

\usepackage{overpic}
\usepackage{a4wide}
\usepackage{mathtools}

\newtheorem{theorem}{Theorem}[section]
\newtheorem{corollary}[theorem]{Corollary}
\newtheorem{prop}[theorem]{Proposition}
\newtheorem{lemma}[theorem]{Lemma}
\theoremstyle{remark}
\newtheorem{definition}[theorem]{Definition}

\newtheorem{example}[theorem]{Example}
\newtheorem{remark}[theorem]{Remark}

\renewcommand{\C}{\mathbb{C}}

\newcommand{\rE}{\mathsf{E}}
\newcommand{\E}{\mathsf{E}}
\newcommand{\EE}{\mathsf{D}}

\newcommand{\diag}{{\rm diag}}

\renewcommand{\G}{\Gamma}

\newcommand{\GL}{\operatorname{GL}}

\newcommand{\ol}{\overline}

\newcommand{\Q}{{\mathbb Q}}
\newcommand{\R}{\mathbb{R}}

\newcommand{\V}{\mathsf{V}}

\newcommand{\Z}{\mathbb{Z}}

\newcommand{\se}{\mathsf{e}}
\newcommand{\sx}{\mathsf{x}}
\newcommand{\A}{\mathsf{A}}
\newcommand{\sG}{\mathsf{G}}
\newcommand{\sZ}{\mathsf{Z}}
\newcommand{\K}{\mathcal{K}}

\newcommand{\new}[1]{{\textcolor{black}{\em #1}}}

\title{Graph coverings and twisted operators}

\author{David Cimasoni}
\address{David Cimasoni -- Universit\'e de Gen\`eve, Suisse}
\email{david.cimasoni@unige.ch}

\author{Adrien Kassel}
\address{Adrien Kassel -- CNRS, \'Ecole Normale Sup\'erieure de Lyon, France}
\email{adrien.kassel@ens-lyon.fr}
\date{\today}

\thanks{D.C. acknowledges partial support of the Swiss NSF  grant 200020-200400.
A.K. acknowledges the hospitality of the Universit\'{e} de Gen\`{e}ve and partial support from the ANR project DIMERS, grant number ANR-18-CE40-0033.}



\begin{document}

\begin{abstract}
Given a graph and a representation of its fundamental group,
there is a naturally associated twisted adjacency operator, uniquely defined up to conjugacy.
The main result of this article is the fact that this
operator behaves in a controlled way under graph covering maps.
When such an operator can be used to enumerate objects, or compute a partition function,
this has concrete implications on the corresponding enumeration problem, or statistical mechanics model.
For example, we show that if~$\widetilde{\G}$ is a finite covering graph of a connected graph~$\G$ endowed
with edge-weights~$\sx=\{\sx_\se\}_{\se}$, then the
spanning tree partition function of~$\G$ divides the one of~$\widetilde{\G}$ in the ring~$\Z[\sx]$.
Several other consequences are obtained, some known, others new.
\end{abstract}

\maketitle


\section{Introduction}

The aim of this article is to present a result of algebraic graph theory, probably known to the experts,
in a fairly self-contained and elementary manner.
This brings into what we believe to be the correct framework several well-known results in combinatorics,
statistical mechanics, and~$L$-function theory, but also provides new ones.
In order to preserve its non-technical nature, we focus in the present article on relatively direct consequences,
leaving the more elaborate implications to subsequent papers, see in particular~\cite{Cimasoni-Klein}.

\medskip

We now explain our main result in an informal way,
referring to Section~\ref{sec:background} for precise definitions
and background, to Theorem~\ref{thm:main} for the complete formal statement,
and to Section~\ref{sub:main} for its proof.

Given a locally finite weighted graph~$\G$ and a representation~$\rho$ of its fundamental group,
one can define a \new{twisted adjacency operator}~$\A_\G^\rho$, see Equation~\eqref{equ:operator},
which is well-defined up to conjugacy.
Consider a covering map~$\widetilde{\G}\to\G$ of finite degree between two connected locally finite graphs.
Via this map, the fundamental group~$\pi_1(\widetilde{\G})$ embeds into~$\pi_1(\G)$. As a consequence,
any representation~$\rho$ of~$\pi_1(\widetilde{\G})$ defines an
\new{induced representation}~$\rho^\#$ of~$\pi_1(\G)$.
Our main result is the fact that the operator~$\A^\rho_{\widetilde{\G}}$
is conjugate to~$\A_\G^{\rho^\#}$.

Let us mention that the existence of a natural isomorphism between the vector spaces on which
the twisted adjacency operators act can be
understood as a chain-complex version of the so-called \new{Eckmann-Shapiro lemma}, originally stated in
group cohomology (see Remark~\ref{rem:Shapiro}).
The interesting part of Theorem~\ref{thm:main}, which we have not been able to find in the literature,
is the fact that the explicitized natural isomorphism conjugates the aptly defined twisted adjacency operators.

As an immediate consequence of this result, we see that the decomposition of~$\rho^\#$ into irreducible representations
leads to a direct sum decomposition of~$\A_\G^{\rho^\#}$, and therefore of~$\A^\rho_{\widetilde{\G}}$.
For example, if~$\rho$ is taken to be the trivial representation,
we readily obtain the fact that~$\A_\G$ is a direct summand of~$\A_{\widetilde{\G}}$, see Corollary~\ref{cor:1}.
(Here, the absence of superscript means that these operators are not twisted,
or twisted by the trivial representation.)
Furthermore, if the covering is normal, then~$\A_{\widetilde{\G}}$ factors as
a direct sum of the operators~$\A_\G^\rho$ twisted by the irreducible representations of the Galois
group of the covering, see Corollary~\ref{cor:2}.

\medskip

Whenever~$\A^\rho_\G$ can be used to enumerate combinatorial objects in~$\G$, or in
an associated graph~$\sG$, these statements have very concrete combinatorial implications.
More generally, if these operators can be used to compute some partition functions of the weighted graph~$(\G,\sx)$,
or of an associated weighted graph~$(\sG,\sx)$, these results have often non-trivial
consequences on the corresponding
models. Several of these implications are well-known, but others are new.
We now state some of them, referring to Section~\ref{sec:app} for details.

There is an obvious place to start, namely the \new{matrix-tree theorem}: the Laplacian~$\Delta_\sG$ allows to
enumerate spanning trees (STs) and rooted spanning forests (RSFs) in~$\sG$.
More generally, if~$\sG=(\V,\E)$ is endowed
with edge-weights~$\sx=\{\sx_\se\}_{\se\in\E}$, then a weighted version of~$\Delta_\sG$ allows to compute the corresponding partition
functions~$\sZ_\mathit{ST}(\sG,\sx)$ and~$\sZ_\mathit{RSF}(\sG,\sx)$,
which can be thought of as elements of the polynomial ring~$\Z[\sx]=\Z[\{\sx_\se\}_{\se\in\E}]$.
Applying Corollary~\ref{cor:1} to the Laplacian~$\A_\G=\Delta_\sG$, we obtain the following result:
if~$\widetilde{\sG}$ is a finite covering graph of a finite connected graph~$\sG$ endowed with
edge-weights~$\sx$, and if~$\tilde{\sx}$ denotes these weights lifted to
the edges of~$\widetilde{\sG}$, then~$\sZ_\mathit{ST}(\sG,\sx)$
divides~$\sZ_\mathit{ST}(\widetilde{\sG},\tilde{\sx})$
and~$\sZ_\mathit{RSF}(\sG,\sx)$ divides~$\sZ_\mathit{RSF}(\widetilde{\sG},\tilde{\sx})$ in the
ring~$\Z[\sx]$.
As an immediate consequence, the number of spanning trees in~$\sG$ divides the number of spanning
trees in~$\widetilde{\sG}$ (a fact first proved by Berman~\cite{Berman} using a different method),
and similarly for rooted spanning forests (to the best of our knowledge, a new result).

Another interesting class of operators is given by the weighted skew-adjacency matrices defined by
Kasteleyn~\cite{Ka1,Ka2} in his study of the \new{dimer model} on surface graphs.
For this model, Corollary~\ref{cor:1} can only be applied to cyclic coverings,
yielding a known result~\cite{Joc,Kuperberg}.
Applying Corollary~\ref{cor:2} to the case of a graph embedded in the torus yields an immediate proof
of the classical fact that the dimer characteristic polynomial behaves multiplicatively under so-called
\new{enlargement of the fundamental domain}~\cite[Theorem~3.3]{KOS}.
However, applying our results to the study of the dimer model on graphs embedded in the Klein bottle
leads to new powerful results, that are harnessed in the parallel article~\cite{Cimasoni-Klein}.

Let us finally mention that our main result can be interpreted as the fact that the operators~$\A_\G^\rho$
satisfy the so-called \new{Artin formalism}, a set of axioms originating from the study of~$L$-series of
Galois field extensions~\cite{Artin24,Artin30}. As a consequence, we obtain several results on
the associated~$L$-series~$L(\G,\sx,\rho)=\det(\operatorname{I}-\A_\G^\rho)^{-1}$, providing a wide generalization
of the results of Stark and Terras~\cite{S-TI,S-TII}, see Section~\ref{sub:Artin}.

\medskip

We conclude this introduction with one final remark.
There are two ways to consider graphs: either as combinatorial objects, or as topological ones (namely~$1$-dimensional CW-complexes). Hence, there are two corresponding ways to define and study the associated fundamental
groups and covering maps. In our pursuit of simplicity, we have chosen the combinatorial one. As a result, we
provide the reader with a brief and dry but self-contained treatment of the required parts of algebraic
topology translated from the topological to the combinatorial category,
see Sections~\ref{sub:graph}--\ref{sub:covering}.

\medskip

This paper is organised as follows. Section~\ref{sec:background} deals with the necessary background material and
claims no originality: we start from scratch with graphs, their fundamental groups and covering maps, before
moving on to connections on graphs, and basics of representation theory of groups.
Section~\ref{sec:main} contains the definition of the twisted operators, our main result with its proof and
corollaries, together with the analogy with algebraic number theory via the Artin formalism.
Finally, Section~\ref{sec:app} deals with the aforementioned combinatorial applications.


\section{Background on graphs and representations}
\label{sec:background}

In this section, we first recall the combinatorial definitions of the
fundamental group of a graph and of graph coverings, adapting
the standard references~\cite{SerreTrees} and~\cite{Hatcher} to our purposes,
see also~\cite{K-N}.
We then proceed with connections on graph vector bundles following~\cite{Kenyon},
and linear representations of groups following~\cite{Serre}.

\subsection{Graphs and directed graphs}
\label{sub:graph}

This first paragraph deals with the elementary concepts of graph and directed graph.
Since there is no universal agreement on the relevant terminology and notation, we
record here these formal definitions following~\cite{SerreTrees}.

\begin{definition}\label{defgraph}
A \new{directed graph} (or \new{digraph})~$\Gamma$ consists of a set~$\V$ of \new{vertices},
a set~$\EE$ of \new{(directed) edges}, together with
maps~$s,t\colon\EE\to\V$ assigning to each edge~$e\in\EE$ its \new{source} vertex~$s(e)\in\V$ and its \new{target} vertex~$t(e)\in\V$.

A \new{graph}~$\Gamma$ consists of sets~$\V,\EE$ and
maps~$s,t\colon\EE\to\V$ as above,
together with an involution of~$\EE$ 
assigning to each
edge~$e\in\EE$ its \new{inverse}~$\ol{e}\in\EE$ such that~$\ol{e}\neq e$ and~$s(\ol{e})=t(e)$.
We let~$\E=\EE/(e\sim\ol{e})$ denote the set of \new{unoriented edges}, and write~$\se\in\E$ for the unoriented
edge corresponding to~$e,\ol{e}\in\EE$.

A (directed) graph is \new{locally finite} if for all~$v\in\V$, the
sets~$\EE_v=\{e\in\EE\,|\,s(e)=v\}$ and~$\EE^v=\{e\in\EE\,|\,t(e)=v\}$ are finite.
It is called \new{finite} if both sets~$\V$ and~$\EE$ are finite. 
\end{definition}

Note that these graphs are not simple in general: we allow multiple edges as well as \emph{self-loops}, i.e. edges~$e$ with~$s(e)=t(e)$.
Note also that in this formalism, graphs are special types of directed graphs. Moreover,
given a directed graph~$\G$, one can build an associated graph (still denoted by~$\G$) by
formally adding an inverse~$\ol{e}$ to each edge~$e\in\EE$.

\medskip

Let us fix a directed graph~$\Gamma$.
A \new{path of length~$n\ge 1$} is a sequence~$\gamma=(e_1,e_2,\ldots, e_n)$ of edges such that~$t(e_i)=s(e_{i+1})$ for all~$i\in\{1,\ldots, n-1\}$.
We shall write~$s(\gamma)=s(e_1)$ and~$t(\gamma)=t(e_n)$ for the source and target of~$\gamma$, respectively.
A \new{path of length~$0$}, or \new{constant path}~$\gamma$, is given by a vertex, which is both
the source and target of~$\gamma$.
A \new{loop} (based at~$v$) is a path~$\gamma$ with~$s(\gamma)=t(\gamma)=v$.

The directed graph~$\Gamma$ is said to be \new{connected} if for any~$v,w\in\V$, there is a path~$\gamma$ with~$s(\gamma)=~v$ and~$t(\gamma)=w$.

\subsection{The fundamental group of a graph}
\label{sub:fundamental}

Let us now assume that~$\Gamma$ is a graph, and fix a vertex~$v\in\V$.

Note that the set of loops based at~$v$ is a monoid
with respect to the concatenation of paths, with
neutral element~$1$ given by the constant path based at~$v$. 
Let us call two loops based at~$v$ (or more generally, two paths with same source and same target)
\new{homotopic} if one can be obtained from the other
by removing or adding loops of the form~$(e,\ol{e})$ along the path.
Then, the set of homotopy classes of loops based at~$v$ forms a group,
with the inverse of~$\gamma=(e_1,\dots,e_n)$ given by~$\ol{\gamma}=(\ol{e}_n,\dots,\ol{e}_1)$.

\begin{definition}\label{deffund}
This group is the \new{fundamental group of the graph~$\Gamma$ based at~$v$}, and is denoted
by~$\pi_1(\Gamma,v)$. 
\end{definition}

If~$\Gamma$ is connected, then the isomorphism class of~$\pi_1(\Gamma,v)$ is easily seen not to depend on
the base vertex~$v$.

By a slight abuse of terminology, we define the \new{fundamental group of a directed graph~$\G$}
as the fundamental group of the associated graph obtained by adding an inverse to each edge of~$\G$.

\medskip

We will make use of the alternative definition of the fundamental group, based on a spanning tree.
Recall that a \new{circuit} (of length~$n\ge 1$) is a loop~$\gamma=(e_1,\dots,e_n)$
such that~$e_{i+1}\neq\ol{e}_i$ for~$i\in\{1,\dots,n-1\}$,~$e_1\neq\ol{e}_n$, and such that the vertices~$t(e_1),\dots,t(e_n)$ are all distinct. A \new{spanning tree} of~$\G$ is a connected non-empty subgraph~$T\subset\G$
without circuits, such that the vertices of~$T$ coincide with the vertices of~$\G$.
Note that the number of vertices and edges in a finite tree satisfy~$|\V(T)|-|\E(T)|=1$.

The \new{fundamental group of the graph~$\Gamma$ based at~$T$}, denoted by~$\pi_1(\G,T)$, is defined as the quotient
of the free group over~$\EE$ by the relations~$\ol{e}=e^{-1}$ for all edges of~$\G$, and~$e=1$ for all edges of~$T$.
If~$\G$ is connected, then it admits a spanning tree, and the groups~$\pi_1(\Gamma,v)$ and~$\pi_1(\G,T)$ are
easily seen to be isomorphic for all~$v\in\V$ and all spanning trees~$T$ of~$\G$.
As a consequence, if~$\G$ is connected and finite, its fundamental group is free
of rank~$\vert\rE\vert-\vert\V\vert+1$.

\subsection{Covering graphs}
\label{sub:covering}

A \new{morphism of digraphs}~$p$ from~$\widetilde{\G}=(\widetilde{\V},\widetilde{\EE},\tilde{s},\tilde{t})$
to~$\G=(\V,\EE,s,t)$ consists of two maps~$p_0:\widetilde{\V}\to\V$ and~$p_1:\widetilde{\EE}\to\EE$ such
that~$s\circ p_1=p_0\circ\tilde{s}$ and~$t\circ p_1=p_0\circ\tilde{t}$.
A \new{morphism of graphs}~$p\colon\widetilde{\G}\to\G$ is a morphism of digraphs which also satisfies the
equality~$p_1(\ol{e})=\ol{p_1(e)}$ for all~$e\in\widetilde{\EE}$.

As one easily checks, a morphism of graphs~$p\colon\widetilde{\G}\to\G$ induces in the obvious way a
homomorphism of groups~$p_*\colon\pi_1(\widetilde{\G},\tilde{v})\to\pi_1(\G,p(\tilde{v}))$ .

\begin{definition}
\label{defcovering}
A \new{covering map} is a morphism of directed graphs~$p\colon\widetilde{\G}\to\G$ with~$p_0\colon\widetilde{\V}\to\V$
surjective, such that for all~$\tilde{v}\in\widetilde{\V}$, the restriction of~$p_1$ defines
bijections~$\widetilde{\EE}_{\tilde{v}}\to \EE_{p(\tilde{v})}$ and~$\widetilde{\EE}^{\tilde{v}}\to \EE^{p(\tilde{v})}$.
In that case, $\widetilde{\G}$ is called a \new{covering digraph} of~$\G$.
\end{definition}

If~$\G$ is a connected digraph and~$p$ is a covering map, then the fibers~$p_0^{-1}(v)$ and~$p_1^{-1}(e)$
have the same cardinality for all $v\in\V$ and $e\in\EE$. This
is called the \new{degree} of the covering.
From now on, we will drop the subscripts in~$p_0$ and~$p_1$ and denote both maps by~$p$.

Note that any morphism of digraphs~$p\colon\widetilde{\G}\to\G$ extends to a unique morphism between the associated graphs
(obtained by adding an inverse to each edge).
Moreover, if the morphism of digraphs is a covering map, then so is the associated morphism of graphs.
In such a case, the graph~$\widetilde{\G}$ is called a \new{covering graph} of~$\G$.

\medskip

The following \new{path lifting property} is a direct consequence of the definitions,
but nevertheless a fundamental feature of a covering map~$p\colon\widetilde{\G}\to\G$.
Given any
path~$\gamma$ in~$\G$ with~$s(\gamma)=v_0$ and any~$\tilde{v}_0\in p^{-1}(v_0)$,
there is a unique path~$\tilde{\gamma}$ in~$\widetilde{\G}$ with~$p(\tilde{\gamma})=\gamma$
and~$s(\tilde{\gamma})=\tilde{v}_0$.
Furthermore, for~$\gamma$ a loop,
the formula~$[\gamma]\cdot\tilde{v}_0=t(\tilde{\gamma})$ defines an action
of~$\pi_1(\G,v_0)$ on~$p^{-1}(v_0)$. If~$\widetilde{\G}$ is connected, then
this action is easily seen to be transitive, with
isotropy group of~$\tilde{v}_0$ equal to~$p_*(\pi_1(\widetilde{\G},\tilde{v}_0))$.
As a consequence, the degree of the covering coincides with
the index of~$p_*(\pi_1(\widetilde{\G},\tilde{v}_0))$ in~$\pi_1(\G,v_0)$.

\medskip

The easy proof of the following lemma is left to the reader.

\begin{lemma}
\label{lemma:covering}
If~$p\colon\widetilde{\G}\to\G$ is a covering map, then:
\begin{enumerate}[(i)]
\item the homomorphism~$p_*\colon\pi_1(\widetilde{\G},\tilde{v})\to\pi_1(\G,p(\tilde{v}))$ is injective;
\item for any~$e\in\EE$, we have the equalities
\[
\{\tilde{t}(\tilde{e})\in\widetilde{\V}\,|\,\tilde{e}\in p^{-1}(e)\}=p^{-1}(t(e))\quad\text{and}\quad
\{\tilde{s}(\tilde{e})\in\widetilde{\V}\,|\,\tilde{e}\in p^{-1}(e)\}=p^{-1}(s(e))\,.
\]
\end{enumerate}
\end{lemma}
 
Let us finally recall that a covering map~$p\colon\widetilde{\G}\to\G$ is said to be \new{normal} (or \new{regular})
if~$p_*(\pi_1(\widetilde{\G},v))$ is a normal subgroup of~$\pi_1(\G,p(v))$. In such a case,
we denote the quotient group by~$G(\widetilde{\G}/\G)$. This is nothing but the group of covering transformations
of this covering map, usually referred to as the \new{Galois group}.

\subsection{Connections on graphs}
\label{sub:bundle}

Following~\cite[Section~3.1]{Kenyon}, let us fix a \new{vector bundle} on a graph~$\G$, i.e.
a finite-dimensional complex vector space~$W$ and the choice of a vector space~$W_v$ isomorphic to~$W$ for each~$v\in\V$.
Such a vector bundle can be identified with~$W_\G\coloneqq \bigoplus_{v\in\V}W_v\simeq W^\V$.

\begin{definition}
\label{def:connection}
A \new{connection} on a vector bundle~$W_\G$ is the choice~$\Phi=(\varphi_e)_{e\in\EE}$ of an
isomorphism~$\varphi_e\colon W_{t(e)}\to W_{s(e)}$ for each~$e\in\EE$, such
that~$\varphi_{\ol{e}}=\varphi_e^{-1}$ for all~$e\in\EE$.

Two connections~$\Phi=(\varphi_e)_{e\in\EE}$ and~$\Phi'=(\varphi'_e)_{e\in\EE}$ are said to be
\new{gauge-equivalent} if there is a family of automorphisms~$\{\psi_v\colon W_v\to W_v\}_{v\in\V}$
such that~$\psi_{s(e)}\circ\varphi_e=\varphi'_e\circ\psi_{t(e)}$ for all~$e\in\EE$.
\end{definition}

Let us fix a base vertex~$v_0\in\V$, a connection~$\Phi$ on a vector bundle~$W_\G$,
and let us denote~$W_{v_0}$ simply by~$W$.
Any loop~$\gamma=(e_1,\dots,e_n)$ based at~$v_0$ gives an
automorphism~$\varphi_{e_1}\circ\dots\circ\varphi_{e_n}$ of~$W$ called the \new{monodromy} of~$\gamma$.
This construction defines a homomorphism
\[
\rho^\Phi\colon\pi_1(\G,v_0)\longrightarrow\operatorname{GL}(W)\,,
\]
i.e. a representation of the fundamental group of~$\G$ in~$W$.

Any representation~$\rho\colon\pi_1(\G,v_0)\to\operatorname{GL}(W)$ is of the form~$\rho^\Phi$ for some connection~$\Phi$: indeed, one can fix a spanning tree~$T\subset\G$ (recall that~$\pi_1(\G,v_0)\simeq\pi_1(\G,T)$),
set~$\varphi_e=\operatorname{id}_W$ for each edge of~$T$ and~$\varphi_e=\rho_e$ for each of the remaining edges of~$\G$.
Furthermore, given two connections~$\Phi$ and~$\Phi'$ on~$W_\G$,
one easily checks that~$\rho^\Phi$ and~$\rho^{\Phi'}$ are conjugate representations
if and only if~$\Phi$ and~$\Phi'$ are gauge-equivalent connections.

In other words, the\new{~$\operatorname{GL}(W)$-character variety}
of~$\pi_1(\G,v_0)$, i.e. the set of conjugacy classes of homomorphisms~$\pi_1(\G,v_0)\to\operatorname{GL}(W)$,
is given by the set of connections on~$W_\G$ up to gauge-equivalence.

\begin{remark}
\label{rem:connection}
The definition of a connection as a family of isomorphisms~$\varphi_e\colon W_{s(e)}\to W_{t(e)}$ seems more natural, but
leads to antihomomorphisms of~$\pi_1(\G,v_0)$. On the other hand, our convention yields
homomorphisms, and is coherent with the definition of a local coefficient system for twisted homology,
see e.g.~\cite[p.~255]{Whi}.
\end{remark}

\subsection{Linear representations of groups}
\label{sub:induced}

We now recall the necessary notation and terminology of linear representations of groups, following~\cite{Serre}.
Throughout this subsection,~$G$ denotes a group.

\medskip

Let us first recall that the \new{degree} of a
representation~$\rho\colon G\to\operatorname{GL}(W)$, denoted by~$\deg(\rho)$, is defined as the dimension of~$W$,
which we always assume to be finite.
The only representation of degree~$0$ is written~$\rho=0$, while the degree~$1$ representation
sending all elements of~$G$ to~$1\in\C^*=\GL(\C)$ is denoted by~$\rho=1$.

Let us now fix two linear representations~$\rho\colon G\to\operatorname{GL}(W)$
and~$\rho'\colon G\to\operatorname{GL}(W')$. The \new{direct sum} of~$\rho$ and~$\rho'$
is the representation~$\rho\oplus\rho'\colon G\to\operatorname{GL}(W\oplus W')$
given by~$(\rho\oplus\rho')_g=\rho_g\oplus\rho_{g'}$.
A representation of~$G$ is said to be \new{irreducible} if it is not the direct sum
of two representations that are both not~$0$.

\medskip

Now, fix a subgroup~$H<G$ of finite index, and a representation~$\rho\colon H\to\operatorname{GL}(W)$.
There is a representation~$\rho^\#\colon G\to\operatorname{GL}(Z)$
which is uniquely determined up to isomorphism by the following two properties.
Let~$R\subset G$ denote a \new{set of representatives} of~$G/H$, i.e. each~$g\in G$ can be written uniquely as~$g=rh\in G$ with~$r\in R$ and~$h\in H$.
\begin{enumerate}[(i)]
\item We have~$Z=\bigoplus_{r\in R}\rho_r^\#(W)$; in particular, the space~$W$ is a subspace of~$Z$.
\item For any~$h\in H$ and~$w\in W$, we have~$\rho_h^\#(w)=\rho_h(w)$.
\end{enumerate}
Indeed, the first property ensures that any element of~$Z$ can be written uniquely as~$\sum_{r\in R}\rho_r^\#(w_r)$
with~$w_r\in W$, while the second one implies that for any~$g\in G, r\in R$ and~$w\in W$, we have~$\rho_g^\#(\rho_r^\#(w))=\rho_{r'}^\#(\rho_h(w))$ where~$gr=r'h\in G$ with~$r'\in R$ and~$h\in H$.

This representation~$\rho^\#\colon G\to\operatorname{GL}(Z)$ is said to be \new{induced} by~$\rho\colon H\to\operatorname{GL}(W)$. 

\begin{example}
\label{ex:ind}
Let us fix a finite index subgroup~$H<G$ and consider the trivial representation~$\rho=1$ of~$H$.
By definition, the induced representation~$\rho^\#\colon G\to\GL(Z)$ is given by the action by left multiplication
of~$G$ on the vector space~$Z$ with basis~$G/H$.
Since~$G$ acts by permutation on the set~$G/H$, which is finite, the subspace
of~$Z$ generated by the sum of these basis elements is fixed by this action.
Therefore, the induced representation splits as~$\rho^\#=1\oplus\rho'$ for some
representation~$\rho'$ of~$G$.

Let us assume further that~$H$ is a normal subgroup of~$G$. In such a case,
the induced representation can be written as~$\rho^\#=\rho_\mathrm{reg}\circ\mathrm{pr}$,
with~$\mathrm{pr}\colon G\to G/H$ the canonical projection and~$\rho_\mathrm{reg}$ the so-called \emph{regular
representation} of~$G/H$.
Since this group is finite, this representation splits as
\[
\rho_\mathrm{reg}=\bigoplus_{\rho\text{ irred.}}\rho^{\oplus\deg(\rho)}\,,
\]
the sum being over all irreducible representations of~$G$ (see~\cite[Section~2.4]{Serre}).
\end{example}


\section{Twisted operators on graph coverings}
\label{sec:main}

This section contains the proof of our main result, Theorem~\ref{thm:main}, which
relates twisted adjacency operators on directed graphs connected by a covering map.
We start in Section~\ref{sub:operator} by defining the relevant twisted operators,
while Section~\ref{sub:main} deals with Theorem~\ref{thm:main}, its proof, and a couple of
corollaries.
Finally, Section~\ref{sub:Artin} shows how this result can be interpreted as
a combinatorial version of the Artin formalism for these operators, yielding consequences on associated~$L$-series.

\subsection{Twisted weighted adjacency operators}
\label{sub:operator}

Fix a locally finite directed graph~$\G=(\V,\EE,s,t)$. Let us assume that it is endowed with
\new{edge-weights}, i.e. a collection~$x=\{x_e\}_{e\in\EE}$ of complex numbers attached to the edges.
The associated \new{weighted adjacency operator}~$\A_\G$ acts on~$\C^\V$ via
\[
(\A_\G f)(v)=\sum_{e\in\EE_v}x_e\,f(t(e))\quad\text{for all~$f\in \C^\V$ and~$v\in\V$}\,.
\]

Adapting~\cite[Section~3.2]{Kenyon} to our purposes, this operator can be twisted by a
representation~$\rho\colon\pi_1(\G,v_0)\to\operatorname{GL}(W)$
in the following way. Fix a vector bundle~$W_\G\simeq W^\V$ and a connection~$\Phi=(\varphi_e)_{e\in\EE}$
representing~$\rho$.

\begin{definition}
\label{def:operator}
The associated \new{twisted weighted adjacency operator}~$\A_\G^\rho$ is the operator on~$W^\V$ given by
\begin{equation}
\label{equ:operator}
(\A^\rho_\G f)(v)=\sum_{e\in\EE_v}x_e\,\varphi_e(f(t(e)))\quad\text{for all~$f\in W^\V$ and~$v\in\V$}\,.
\end{equation}
\end{definition}

Several remarks are in order.

\begin{remark}
\label{rem:twisted}
\begin{enumerate}[(i)]
\item We make a slight abuse of notation in the sense that the
operator~$\A_\G^\rho$ depends on the choice of a connection~$\Phi$ and is
therefore not entirely determined by~$\rho$.
However, by
Section~\ref{sub:bundle}, conjugate representations are given by gauge equivalent connections.
Furthermore, the corresponding twisted operators are conjugated by an element of~$\GL(W)^\V\subset\GL(W^\V)$.
Therefore, the conjugacy class of~$\A_\G^\rho$ only depends on the conjugacy class of~$\rho$.
\item If a representation~$\rho$ is given by the direct sum of~$\rho_1$ and~$\rho_2$, then the
operator~$\A_\G^\rho$ is conjugate to~$\A_\G^{\rho_1}\oplus \A_\G^{\rho_2}$.
\item The operator~$\A^1_\G$ is nothing but the untwisted operator~$\A_\G$.
\end{enumerate}
\end{remark}

\medskip

Obviously, any given directed graph~$\G$ defines a single untwisted operator~$\A_\G$,
so it might seem at first sight that the applications of our results will be rather limited.
However, there are many natural assignments~$\sG\mapsto\G$ mapping a directed graph~$\sG$ to
another directed graph~$\G$ so that~$\A_\G$ provides a new operator on~$\sG$.
Moreover, if there is a natural homomorphism~$\alpha\colon\pi_1(\G)\to\pi_1(\sG)$, then
a~$\rho$-twisted version of this new operator can be understood as~$\A_\G^{\rho\circ\alpha}$.
Finally, if the assignment~$\sG\mapsto\G$ preserves covering maps, then our results apply to these new
twisted operators as well.

We now give three explicit examples of such natural maps~$\sG\mapsto\G$, claiming no exhaustivity.
It is easy indeed to find additional interesting ones, e.g. the \new{Fisher correspondance} used
in the study of the Ising model~\cite{Fisher}.

\begin{example}
\label{ex:Delta}
Let~$\mathsf{G}=(\V(\sG),\E(\sG))$ be a graph
endowed with \new{symmetric} edge-weights,
i.e. labels~$\sx=(\sx_\se)_{\se\in\E(\sG)}$ associated to its unoriented edges.
Consider the associated graph~$\G=(\V,\E)$ defined by~$\V=\V(\sG)$ and~$\E=\E(\sG)\cup\V(\sG)$,
where each~$v\in\V(\sG)$ produces an unoriented edge~$\{v,\ol{v}\}$ and the
source, target, and involution maps of~$\G$ are given by extending the ones of~$\sG$
via~$s(v)=t(v)=v$ for all~$v\in\V$.
(Concretely, the graph~$\G$ is obtained from~$\sG$ by adding a self-loop at each vertex.)
Also, extend the edge-weights on~$\sG$ to symmetric edge-weights on~$\G$
via~$\sx_v=-\sum_{e\in\EE(\sG)_v}\sx_e$. Then, the corresponding weighted adjacency operator~$\A_\G$
is (the opposite of) the \new{Laplacian}~$\Delta_\sG$ on~$\sG$.
It can be used to count spanning trees of~$\sG$ --- this is the celebrated \new{matrix-tree theorem} ---
but also rooted spanning forests, see Section~\ref{sub:tree}.

Note that there is a natural homomorphism~$\alpha\colon\pi_1(\G,v_0)\to\pi_1(\sG,v_0)$
mapping all the newly introduced self-loops to the neutral element. Given any representation~$\rho$ of~$\pi_1(\sG,v_0)$,
the associated twisted operator~$\A^{\rho\circ\alpha}_\G$
is the \new{vector bundle Laplacian}~$\Delta^\rho_\sG$ of~\cite{Kenyon}.
When the representation~$\rho$ takes values in~$\C^*$ or~$\operatorname{SL}_2(\C)$, then~$\Delta^\rho_\sG$
can be used to study cycle-rooted spanning forests~\cite{Forman,Kenyon},
while representations of higher degree yield more involved combinatorial objects.
\end{example}

\begin{example}
\label{ex:Kasteleyn}
Let~$\G$ be a graph endowed with symmetric edge-weights~$\sx=(\sx_\se)_{\se\in\E}$.
Fix an orientation of the edges of~$\G$ and consider the same graph~$\G$ endowed with the anti-symmetric
edge-weights~$x=\{x_e\}_{e\in\EE}$ given by~$x_e=\sx_\se$ if the orientation of~$e\in\EE$ agrees with
the fixed orientation, and~$x_e=-\sx_\se$ otherwise. Then, the operator~$\A_\G$ is a weighted skew-adjacency
operator that was used by Kasteleyn~\cite{Ka1,Ka2} and many others in the study of the 2-dimensional
dimer and Ising models,
see Section~\ref{sub:dimers}.
Such operators twisted by~$\operatorname{SL}_2(\C)$-representations
are also considered by Kenyon in his study of the double-dimer model~\cite{Kenyon:dd}. 
\end{example}

\begin{example}
\label{ex:S-T}
Let us start with a graph~$\mathsf{G}=(\V(\sG),\EE(\sG),s_\sG,t_\sG,i)$ endowed with symmetric edge-weights~$\sx=\{\sx_\se\}_{\se}$, and consider the
associated \new{directed line graph}~$\G=(\V,\EE,s,t)$ defined by
\[
\V=\EE(\sG),\quad\EE=\{(e,e')\in\V\times\V\,|\,t_\sG(e)=s_\sG(e')\text{ but }e'\neq\ol{e}\},\quad s(e,e')=e,\quad t(e,e')=e',
\]
and endowed with the edge-weights~$x=\{x_{e,e'}\}_{(e,e')\in\EE}$ defined by~$x_{e,e'}=\sx_\se$.
Then, the operator~$\operatorname{I}-\A_\G$
is considered by Stark and Terras~\cite{S-TI,S-TII} in their study of prime cycles (see Section~\ref{sub:Artin}),
while a similar operator is defined by Kac and Ward~\cite{K-W} in their exploration of the planar Ising model
(see Section~\ref{sub:dimers}).
Note also that there is a natural homomomorphism~$\alpha\colon\pi_1(\G,e_0)\to\pi_1(\sG,s(e_0))$, so any
representation~$\rho\colon\pi_1(\sG,v_0)\to\GL(W)$ defines a twisted
operator~$\operatorname{I}-\A^{\rho\circ\alpha}_\G$.
\end{example}

\subsection{The main result}
\label{sub:main}

We are finally ready to state and prove our main theorem.

Let~$\G=(\V,\EE,s,t)$ be a locally finite connected directed graph
with weights~$x=(x_e)_{e\in\EE}$,
and let~$p\colon\widetilde{\G}\to\G$ be a covering map of finite degree~$d$,
with~$\widetilde{\G}=(\widetilde{\V},\widetilde{\EE},\tilde{s},\tilde{t})$ connected.
The weights~$x$ on~$\G$ lift to weights~$\tilde{x}$ on~$\widetilde{\G}$ via~$\tilde{x}_{\tilde{e}}\coloneqq x_{p(\tilde{e})}$
for all~$\tilde{e}\in\widetilde{\EE}$, so~$\widetilde{\G}$ is a
weighted directed graph, which is locally finite.

Fix base vertices~$v_0\in\V$ and~$\tilde{v}_0\in p^{-1}(v_0)$,
and recall from Lemma~\ref{lemma:covering} that~$p$ induces an
injection~$p_*\colon\pi_1(\widetilde{\G},\tilde{v}_0)\to\pi_1(\G,v_0)$ between the fundamental groups of the associated
graphs, so~$\pi_1(\widetilde{\G},\tilde{v}_0)$ can
be considered as a subgroup of~$\pi_1(\G,v_0)$ of index~$d$.
Therefore, as explained in
Section~\ref{sub:induced},
any representation~$\rho\colon\pi_1(\widetilde{\G},\tilde{v}_0)\to\GL(W)$ induces a
representation ~$\rho^\#\colon\pi_1(\G,v_0)\to\GL(Z)$.

\begin{theorem}
\label{thm:main}
For any covering map of connected directed graphs~$p\colon\widetilde{\G}\to\G$ as above and any
representation~$\rho$ of~$\pi_1(\widetilde{\G},\tilde{v}_0)$,
the operators~$\A_{\widetilde{\G}}^\rho$ and~$\A_\G^{\rho^\#}$ are conjugate.
\end{theorem}

\begin{remark}
\label{rem:Shapiro}
The existence of a natural isomorphism~$W^{\widetilde{\V}}\simeq Z^\V$ is a chain-complex version
of the \new{Eckmann-Shapiro Lemma}, traditionally stated in the context of group (co)homology
(see e.g.~\cite[p.~73]{Brown}).
Moreover, the tensor-product definition of the induced representation (see~\cite[Chapter~7]{Serre}) makes
the existence of this isomorphism a routine check.
The interesting part of Theorem~\ref{thm:main} is the explicit form of this natural isomorphism
in our setting, which turns out to conjugate the relevant twisted adjacency operators.
\end{remark}

Before giving the proof of Theorem~\ref{thm:main}, we present a couple of consequences.

\begin{corollary}
\label{cor:1}
If~$\widetilde{\G}$ is a connected covering digraph of~$\G$ of finite degree,
then~$\A_{\widetilde{\G}}$ is conjugate to~$\A_\G\oplus\A_\G^{\rho'}$ for some
representation~$\rho'$ of~$\pi_1(\G,v_0)$.
\end{corollary}

\begin{proof}
Applying Theorem~\ref{thm:main} to the trivial representation~$\rho=1$ of~$\pi_1(\widetilde{\G},\tilde{v}_0)$,
we get that~$\A^\rho_{\widetilde{\G}}=\A_{\widetilde{\G}}$ is conjugate to~$\A_\G^{\rho^\#}$,
with~$\rho^\#$ the induced representation of~$\pi_1(\G,v_0)$.
By the first part of Example~\ref{ex:ind}, it splits as~$\rho^\#=1\oplus\rho'$ for some
representation~$\rho'$ of~$\pi_1(\G,v_0)$. The statement now follows from
the second and third points of Remark~\ref{rem:twisted}.
\end{proof}

\begin{corollary}
\label{cor:2}
If~$\widetilde{\G}\to\G$ is a normal covering map of finite degree with~$\widetilde{\G}$ connected,
then~$\A_{\widetilde{\G}}$ is conjugate to
\[
\bigoplus_{\rho\text{ irred.}}\left(\A_\G^{\rho\circ\mathrm{pr}}\right)^{\oplus\deg(\rho)}\,,
\]
where the direct sum is over all irreducible representations of~$G(\widetilde{\G}/\G)$,
and~$\mathrm{pr}$ stands for the canonical projection of~$\pi_1(\G,v_0)$ onto~$\pi_1(\G,v_0)/p_*(\pi_1(\widetilde{\G},\tilde{v}_0))=G(\widetilde{\G}/\G)$.
\end{corollary}

\begin{proof}
This is a direct consequence of Theorem~\ref{thm:main} applied to the trivial representation~$\rho=1$
of~$\pi_1(\widetilde{\G},\tilde{v}_0)$
together with Example~\ref{ex:ind} and the second point of Remark~\ref{rem:twisted}.
\end{proof}

\begin{proof}[Proof of Theorem~\ref{thm:main}]
Let~$p\colon\widetilde{\G}\to\G$ be a covering map sending the base
vertex~$\tilde{v}_0$ of~$\widetilde{\G}$
to the base vertex~$v_0$ of~$\G$,
with~$\G=(\V,\EE,s,t)$ a locally finite and connected directed graph endowed with edge-weights~$x=(x_e)_{e\in\EE}$,
and~$\widetilde{\G}=(\widetilde{\V},\widetilde{\EE},\tilde{s},\tilde{t})$ a (locally finite) connected directed graph
endowed with the lifted edge-weights~$\tilde{x}=(\tilde{x}_{\tilde{e}})_{\tilde{e}\in\widetilde{\EE}}$
defined by~$\tilde{x}_{\tilde{e}}=x_{p(\tilde{e})}$.
As always, we use the same notation~$\widetilde{\G},\G$ for the directed graphs and for the associated graphs.

Let~$\rho\colon\pi_1(\widetilde{\G},\tilde{v}_0)\to\GL(W)$ be a representation,
and let~$\widetilde{\Phi}=(\widetilde{\varphi}_{\tilde{e}})_{\tilde{e}\in\widetilde{\EE}}$ be an arbitrary connection on a vector bundle~$W_{\widetilde{\Gamma}}=\bigoplus_{\tilde{v}\in\widetilde{\V}}W_{\tilde{v}}$ such
that~$\rho^{\widetilde{\Phi}}=\rho$ (recall Section~\ref{sub:bundle}).
Consider the vector bundle on~$\Gamma$ given by~$Z_\Gamma=\bigoplus_{v\in\V}Z_v$, where
\[
Z_v\coloneqq\bigoplus_{\tilde{v}\in p^{-1}(v)}W_{\tilde{v}}\,.
\]
This definition leads to the equality~$Z_\Gamma=\bigoplus_{v\in\V}\bigoplus_{\tilde{v}\in p^{-1}(v)}W_{\tilde{v}}=\bigoplus_{\tilde{v}\in\widetilde{\V}}W_{\tilde{v}}=W_{\widetilde{\Gamma}}$.

Next, consider the connection~$\Phi=(\varphi_e)_{e\in\EE}$
defined by
\[
\varphi_e\coloneqq\bigoplus_{\tilde{e}\in p^{-1}(e)}\widetilde{\varphi}_{\tilde{e}}\colon\bigoplus_{\tilde{e}\in p^{-1}(e)}W_{\tilde{t}(\tilde{e})}\longrightarrow\bigoplus_{\tilde{e}\in p^{-1}(e)}W_{\tilde{s}(\tilde{e})}\,.
\]
Note that the second point of Lemma~\ref{lemma:covering} gives~$\bigoplus_{\tilde{e}\in p^{-1}(e)}W_{\tilde{t}(\tilde{e})}=\bigoplus_{\tilde{v}\in p^{-1}(t(e))}W_{\tilde{v}}=Z_{t(e)}$ and
similarly for~$Z_{s(e)}$. Therefore, the formula displayed above defines a map~$\varphi_e\colon Z_{t(e)}\to Z_{s(e)}$, and~$\Phi$ is a connection on the vector bundle~$Z_\Gamma$.

Let us denote by~$\A^\Phi_\Gamma$ the operator~$\A_\Gamma$ twisted by the  connection~$\Phi$,
and similarly for~$\A_{\widetilde{\Gamma}}^{\widetilde{\Phi}}$.
For any~$f\in Z_\Gamma=W_{\widetilde{\Gamma}}$ and~$e\in\EE_v$, one can
write
\[
f(t(e))=\bigoplus_{\tilde{v}\in p^{-1}(t(e))}f(\tilde{v})=\bigoplus_{\tilde{e}\in p^{-1}(e)}f(\tilde{t}(\tilde{e}))
\]
by Lemma~\ref{lemma:covering}. For any~$v\in\V$, this leads to
\[
(\A^\Phi_\G f)(v)
	=\sum_{e\in\EE_v}x_e\,\varphi_e(f(t(e)))
	=\sum_{e\in\EE_v}x_e\bigoplus_{\tilde{e}\in p^{-1}(e)}\widetilde{\varphi}_{\tilde{e}}(f(\tilde{t}(\tilde{e})))
	=\sum_{e\in\EE_v}\bigoplus_{\tilde{e}\in p^{-1}(e)}\tilde{x}_{\tilde{e}}\,\widetilde{\varphi}_{\tilde{e}}(f(\tilde{t}(\tilde{e})))\,.
\]
Since~$p$ is a morphism of graphs, this sum is equal to
\[
\bigoplus_{\tilde{v}\in p^{-1}(v)}\sum_{\tilde{e}\in\widetilde{\EE}_{\tilde{v}}}\tilde{x}_{\tilde{e}}\,\widetilde{\varphi}_{\tilde{e}}(f(\tilde{t}(\tilde{e})))=\bigoplus_{\tilde{v}\in p^{-1}(v)}(\A_{\widetilde{\Gamma}}^{\widetilde{\Phi}}f)(\tilde{v})=(\A_{\widetilde{\Gamma}}^{\widetilde{\Phi}}f)(v)\in\bigoplus_{\tilde{v}\in p^{-1}(v)}W_{\tilde{v}}=Z_v\,.
\]
In conclusion, the explicit operator~$\A_{\widetilde{\Gamma}}^{\widetilde{\Phi}}$ representing~$\A^\rho_{\widetilde{\G}}$ coincides with~$\A_\G^\Phi$. Therefore, we are left with the proof that the connection~$\Phi$ on~$Z_\G$ is such that~$\rho^\Phi=\rho^\#$.

To do so, let us check that~$\rho^\Phi$ satisfies the two defining properties of the representation~$\rho^\#\colon\pi_1(\G,v_0)\to\GL(Z)$ induced by~$\rho\colon\pi_1(\widetilde{\G},\tilde{v}_0)\to\GL(W)$, as stated in Section~\ref{sub:induced}.
Consider a loop~$\gamma$ in~$\G$ based at~$v_0$. For any~$\tilde{v}\in p^{-1}(v_0)$,
the automorphism~$\rho_{[\gamma]}^\Phi$ maps the elements
of~$W_{\tilde{v}}\subset Z_{v_0}=Z$ to the component of~$Z$
corresponding to the endpoint of the lift of~$\gamma$ starting at~$\tilde{v}$. In other words, and with the notation of Section~\ref{sub:covering}, we have the equality~$\rho_{[\gamma]}^\Phi(W_{\tilde{v}})=W_{[\gamma]\cdot\tilde{v}}$.
Since~$\widetilde{\G}$ is connected, the action of~$\pi_1(\G,v_0)$ on~$p^{-1}(v_0)$ is transitive,
with the isotropy group of~$\tilde{v}_0$ equal to~$p_*(\pi_1(\widetilde{\G},\tilde{v}_0))$.
Therefore, we see that the space~$Z=\bigoplus_{\tilde{v}\in p^{-1}(v_0)}W_{\tilde{v}}$ is indeed the direct sum of the images of~$W_{\tilde{v}_0}$ by any set of representatives of~$\pi_1(\G,v_0)/p_*(\pi_1(\widetilde{\G},\tilde{v}_0))$.
Finally, for any loop~$\widetilde\gamma$ in~$\widetilde{\G}$ based at~$\tilde{v}_0$ and any vector~$w\in W_{\tilde{v}_0}$,
we have~$\rho_{p_*([\widetilde\gamma])}^\Phi(w)=\rho_{[\widetilde\gamma]}^{\widetilde{\Phi}}(w)=\rho_{[\widetilde\gamma]}(w)$ by definition, thus showing the second point. This concludes the proof.
\end{proof}

\subsection{The Artin formalism for graphs}
\label{sub:Artin}

In his foundational work in algebraic number theory~\cite{Artin24,Artin30}, Artin associates an~$L$-series
to any Galois field extension endowed with a representation of its Galois group.
He shows that these~$L$-series satisfy four axioms, the so-called \new{Artin formalism}
(see~\cite[Chapter~XII.2]{Lang} for a modern account).
Since then, analogous axioms have been shown to hold for~$L$-series in topology~\cite{Lang56}, in
analysis~\cite{Lang94}, and for some~$L$-series associated to finite graphs~\cite{S-TII}.

The aim of this subsection is to explain how Theorem~\ref{thm:main} can be interpreted as (the non-trivial
part of) an Artin formalism for graphs. We also show that our approach allows for wide generalisations of the results of Stark and Terras~\cite{S-TI,S-TII}.

\medskip

In what follows, for simplicity, we omit the basepoint when we write fundamental groups. Recall from Section~\ref{sub:operator} that to any weighted locally finite directed graph~$\G=(\V,\EE,s,t)$
endowed with a representation~$\rho\colon\pi_1(\G)\to\GL(W)$, we associate
a twisted weighted adjacency operator~$\A_\G^\rho$ in~$\operatorname{End}(W^\V)$, well defined up to
conjugation.
Recall also that a covering~$p\colon\widetilde{\G}\to\G$ is said to be normal if~$\pi_1(\widetilde{\G})$
is a normal subgroup of~$\pi_1(\G)$. In this case, the quotient group~$G=\pi_1(\G)/\pi_1(\widetilde{\G})$
is called the Galois group of the covering.
Given a normal covering~$p\colon\widetilde{\G}\to\G$ (also simply written as~$\widetilde{\G}/\G$) and a
representation~$\rho\colon G\to\GL(W)$ of its Galois group, one can form
the representation~$\rho\circ\mathrm{pr}\colon\pi_1(\G)\to G\to\GL(W)$ of~$\pi_1(\G)$.
We denote by~$\mathcal{O}(\widetilde{\G}/\G,\rho)=\left[\A_\G^{\rho\circ\mathrm{pr}}\right]\in\operatorname{End}(W^\V)/\GL(W)^\V$ the conjugacy class of the associated twisted adjacency operator.

\begin{prop}
\label{prop:Artin}
The map which to a normal covering~$\widetilde{\G}/\G$ and a representation~$\rho$ of its Galois group
associates the class of operators~$\mathcal{O}(\widetilde{\G}/\G,\rho)$ satisfies the following four axioms.
\begin{enumerate}[1.]
\item~$\mathcal{O}(\widetilde{\G}/\G,1)=[\A_\G]$, the untwisted adjacency operator on~$\G$.
\item Given any two representations~$\rho_1$ and~$\rho_2$ of~$G$,
we have
\[
\mathcal{O}(\widetilde{\G}/\G,\rho_1\oplus\rho_2)=\mathcal{O}(\widetilde{\G}/\G,\rho_1)\oplus \mathcal{O}(\widetilde{\G}/\G,\rho_2)\,.
\]
\item If~$H$ is a normal subgroup of~$G$ and~$\ol{\G}=H\backslash\widetilde{\G}$ denotes the corresponding covering
of~$\G$, then for any representation~$\rho$ of~$G/H$, we have
\[
\mathcal{O}(\ol{\G}/\G,\rho)=\mathcal{O}(\widetilde{\G}/\G,\rho\circ\pi)\,,
\]
where~$\pi\colon G\to G/H$ denotes the canonical projection.
\item If~$H$ is a subgroup of~$G$ and~$\ol{\G}=H\backslash\widetilde{\G}$, then for any
representation~$\rho$ of~$H$, we have
\[
\mathcal{O}(\widetilde{\G}/\ol{\G},\rho)=\mathcal{O}(\widetilde{\G}/\G,\rho^\#)\,,
\]
where~$\rho^\#$ is the representation of~$G$ induced by~$\rho$.
\end{enumerate}
\end{prop}

\begin{proof}
The first and second points are reformulations of the trivial Remarks~\ref{rem:twisted}~(ii) and~(iii),
while the third point follows from the fact that the
composition~$\pi\circ\mathrm{pr}\colon\pi_1(\G)\to G/H$ coincides with the 
canonical projection of~$\pi_1(\G)$ onto~$\pi_1(\G)/\pi_1(\ol{\G})$.
As for the last point, let~$\ol{p}\colon\ol{\G}\to\G$ denote the relevant covering map,
and~$\ol{\mathrm{pr}}$ the canonical projection of~$\pi_1(\ol{\G})$\
onto~$\pi_1(\ol{\G})/\pi_1(\widetilde{\G})=~H$. By naturality, the composition of~$\ol{\mathrm{pr}}$ 
with the inclusion of~$H$ in~$G$ coincides with~$\mathrm{pr}\circ\ol{p}_*$. Therefore,
the representation induced by~$\rho\circ\ol{\mathrm{pr}}$ coincides with~$\rho^\#\circ\mathrm{pr}$.
The fourth point is now a formal consequence of Theorem~\ref{thm:main}:
\[
\mathcal{O}(\widetilde{\G}/\ol{\G},\rho)=[\A_{\ol{\G}}^{\rho\circ\ol{\mathrm{pr}}}]
=[\A_\G^{(\rho\circ\ol{\mathrm{pr}})^\#}]=[\A_\G^{\rho^\#\circ\mathrm{pr}}]=\mathcal{O}(\widetilde{\G}/\G,\rho^\#)\,.\qedhere
\]
\end{proof}

\medskip

With the~$L$-series of~\cite{S-TII} in mind, it is natural to
consider~$\det(\operatorname{I}-\A_\G^{\rho\circ\mathrm{pr}})^{-1}$ as the object of study.
The fact that these~$L$-series satisfy the Artin formalism follows from the proposition above.

Actually, our method easily yields results on more general~$L$-series, as follows.
Let us fix a map associating to a weighted graph~$(\sG,\sx)$
a weighted directed graph~$(\G,x)$, as in Examples~\ref{ex:Delta}--\ref{ex:S-T}. Formally, we want this assignment to preserve
the ingredients of Theorem~\ref{thm:main}:
a covering map~$\widetilde{\sG}\to\sG$ of locally finite connected graphs is sent to a covering
map~$\widetilde{\G}\to\G$ of locally finite connected digraphs, and there is a natural group
homomorphism~$\alpha\colon\pi_1(\G)\to\pi_1(\sG)$.
Given any representation~$\rho$ of~$\pi_1(\sG)$, we can now consider the~$L$-series
\[
L(\sG,\sx,\rho)=\det(\operatorname{I}-\A_\G^{\rho\circ\alpha})^{-1}\in\C[\![\sx]\!]\,.
\]
By the Amitsur formula (see~\cite{Ami,R-S}), it can be written as
\[
L(\sG,\sx,\rho)=\prod_{[\gamma]}\det(1-x(\gamma)\rho_\gamma)^{-1}\,,
\]
where the product is over all loops~$\gamma$ in~$\G$ that cannot be expressed
as~$\delta^\ell$ for some path~$\delta$ and integer~$\ell>1$,
loops considered up to change of base vertex.
Also,~$x(\gamma)$ denotes the product of the weights of the edges of~$\gamma$, while~$\rho_\gamma$ is
the monodromy of the loop~$\gamma$. (Note that changing the base point yields a conjugate monodromy,
so~$\det(1-x(\gamma)\rho_\gamma)$ is well defined.)
Of course, these loops in~$\G$ correspond to some class of loops in~$\sG$, a class which depends on the way~$\G$ is obtained from~$\sG$.
But for any such assignment, the results of Section~\ref{sub:main} have straightforward implications on the corresponding~$L$-series,
and on the corresponding class of loops in~$\sG$.

For concreteness, let us focus on the directed line graph assignment~$\sG\mapsto\G$ described in Example~\ref{ex:S-T}.
The corresponding twisted weighted operator~$\operatorname{I}-\A_\G^{\rho\circ\alpha}$ coincides with the operator
considered in~\cite[Theorem~7]{S-TII}, where the authors restrict themselves
to representations of a finite quotient of~$\pi_1(\sG)$, i.e. representations of the Galois group of a finite cover of~$\sG$.
In the expression displayed above, the product is over so-called \new{prime cycles} in~$\sG$,
i.e. equivalence classes of cyclic loops in~$\sG$ that do not contain a subpath of the form~$(e,\ol{e})$ and that cannot be expressed as
the power of a shorter loop.
In the special case when~$\rho$ factorises through a finite quotient of~$\pi_1(\sG)$,
this is what Stark and
Terras define as the \new{multiedge Artin~$L$-function} of~$\sG$, an object extending
several other functions introduced in~\cite{S-TI,S-TII}.

The theory of Section~\ref{sec:main} applied to~$\A_\G^{\rho\circ\alpha}$ now allows us to easily extend their results
to this more general~$L$-function.
For example, our Corollary~\ref{cor:1} shows that if~$\widetilde{\sG}$
is a finite connected covering graph of a connected graph~$\sG$, then~$L(\sG,\sx,1)^{-1}$
divides~$L(\widetilde{\sG},\tilde{\sx},1)^{-1}$, extending Corollary~1 of~\cite[Theorem~3]{S-TI}.
Also, our Corollary~\ref{cor:2} recovers the corollary of~\cite[Proposition~3]{S-TII},
while our Theorem~\ref{thm:main} extends Theorem~8 of~\cite{S-TII}.
Finally, expanding the equality~$\log L(\widetilde{\sG},\tilde{\sx},\rho)=\log L(\sG,\sx,\rho^\#)$
yields an extension of the technical Lemma~7 of~\cite{S-TII} to more general covers and representations.

\medskip

We conclude this section by recalling that our approach immediately yields similar results for any
assignment~$\sG\mapsto\G$ preserving covering maps.


\section{Combinatorial applications}
\label{sec:app}

Each time the determinant of an operator counts combinatorial objects,
Theorem~\ref{thm:main} and Corollaries~\ref{cor:1} and~\ref{cor:2} have combinatorial implications.
This is the case for the operators given in Examples~\ref{ex:Delta} and~\ref{ex:Kasteleyn},
whose determinants count spanning trees and perfect matchings, respectively.
We explain these applications in Sections~\ref{sub:tree} and~\ref{sub:dimers}.
We also briefly enumerate additional applications in Section~\ref{sub:further}.

\subsection{Spanning trees and rooted spanning forests}
\label{sub:tree}

Our first combinatorial application relies on a slightly generalised version of the
matrix-tree theorem, that we now recall.

Let~$\sG=(\V,\E)$ be a finite graph endowed with symmetric
weights~$\sx=\{\sx_\se\}_{\se\in\E}$, that we consider as formal variables.
Let~$\Delta_\sG$ be the associated Laplacian, acting on~$\C^{\V}$ via
\[
\Delta_\sG f(v)=\sum_{e\in\EE_v}\sx_e(f(v)-f(t(e)))
\]
for~$f\in\C^{\V}$ and~$v\in\V$.
Set~$n\coloneqq |\V|$, and consider the characteristic polynomial in~$|\E|+1$ variables
\[
P_\sG(\lambda)\coloneqq \det(\lambda\operatorname{I}-\Delta_\sG)=\sum_{i=0}^n c_i\,\lambda^i\in\Z[\sx,\lambda]\,.
\]
Then, the coefficient~$c_i\in\Z[\sx]$ admits the combinatorial interpretation
\[
(-1)^{n-i}\,c_i=\sum_{F\subset\sG,\;|\pi_0(F)|=i}\phi(F)\prod_{\se\in\E(F)}\sx_\se\,,
\]
where the sum is over all spanning forests~$F$ in~$\sG$ with~$i$ connected components
(or equivalently, with~$n-i$ edges), and~$\phi(F)\in\Z_+$ denotes
the number of possible \new{roots} of~$F$: if~$F=\bigsqcup_j T_j$ denotes the decomposition
of~$F$ into connected components, then~$\phi(F)=\prod_j|\V(T_j)|$.

For example, there is a unique spanning forest~$F$ in~$\sG$ with~$n$ connected components
(given by the vertices of~$\sG$), it admits a unique root, leading to
the expected value~$c_n=1$. As additional reality checks, we have the
values~$-c_{n-1}=2\sum_{\se\in\E}\sx_\se$ and~$c_0=\det(\Delta_\sG)=0$.
Finally, since connected spanning forests coincide with spanning trees,
and all spanning trees admit exactly~$n$ roots, we have
\[
(-1)^{n-1}\,c_1=n\sum_{T\subset\sG}\prod_{\se\in\E(T)}\sx_\se\,,
\]
the sum being over all spanning trees of~$\sG$. This
latter result is nothing but Kirchoff's matrix-tree theorem.

\begin{remark}
\label{rem:C-L}
This result can be derived from the (usual version of the) matrix-tree theorem
applied to the graph obtained from~$\sG$ by adding one vertex connected to each vertex of~$\sG$ by an edge
of weight~$-\lambda$.

Let us also mention that this result was obtained by Chung and Langlands
in the context of graphs endowed with vertex-weights rather than edge-weights~\cite{C-L}.
Theorem~\ref{thm:main} trivially extends to graphs endowed with
vertex-weights (in addition to edge-weights), and it is a routine task to adapt the results
of the present subsection to this more general case.
\end{remark}

\begin{definition}
\label{def:partition}
The \new{spanning tree partition function} of a weighted graph~$(\sG,\sx)$ is
\[
\sZ_{\mathit{ST}}(\sG,\sx)\coloneqq \sum_{T\subset\sG}\prod_{\se\in\E(T)}\sx_\se\,,
\]
the sum being over all spanning trees in~$\sG$. Similarly,
the \new{rooted spanning forest partition function} of~$(\sG,\sx)$ is
\[
\sZ_{\mathit{RSF}}(\sG,\sx)\coloneqq \sum_{F\subset\sG}\phi(F)\prod_{\se\in\E(F)}\sx_\se\,,
\]
the sum being over all spanning forests in~$\G$.
\end{definition}

Note that if one sets all the weights to~$1$, then~$\sZ_{\mathit{ST}}(\sG,1)$
is the number of spanning trees in~$\sG$,
while~$\sZ_{\mathit{RSF}}(\sG,1)$ counts the number of rooted spanning forests in~$\sG$.

\begin{theorem}
\label{thm:tree}
Let~$\widetilde{\sG}$ be a finite covering graph of a finite connected graph~$\sG$ endowed with
edge-weights~$\sx=\{\sx_\se\}_{\se\in\E}$, and let~$\tilde{\sx}$ denote these weights lifted to
the edges of~$\widetilde{\sG}$. Then~$\sZ_{\mathit{ST}}(\sG,\sx)$
divides~$\sZ_{\mathit{ST}}(\widetilde{\sG},\tilde{\sx})$
and~$\sZ_{\mathit{RSF}}(\sG,\sx)$ divides~$\sZ_{\mathit{RSF}}(\widetilde{\sG},\tilde{\sx})$ in the
ring~$\Z[\sx]$.
\end{theorem}

This immediately leads to the following corollary. The first point is known since the work of Berman
(see~\cite[Theorem~5.7]{Berman}), while the second one appears to be new.

\begin{corollary}
Let~$\widetilde{\sG}$ be a finite covering graph of a finite connected graph~$\sG$.
\begin{enumerate}[(i)]
\item The number of spanning trees in~$\sG$ divides the number of spanning trees in~$\widetilde{\sG}$.
\item The number of rooted spanning forests in~$\sG$ divides the number of rooted spanning forests in~$\widetilde{\sG}$.\qed
\end{enumerate}
\end{corollary}

\begin{proof}[Proof of Theorem~\ref{thm:tree}]
First note that~$\sZ_{\mathit{RSF}}(\widetilde{\sG},\tilde\sx)$ is multiplicative with respect to connected sums
while~$\sZ_{\mathit{ST}}(\widetilde{\sG},\tilde\sx)$ vanishes for~$\widetilde{\sG}$ not connected.
Therefore, it can be assumed that~$\widetilde{\sG}$ is connected.
Let~$\widetilde{\sG}\to\sG$ be a covering map between two finite connected graphs,
with edge-weights~$\sx$ on~$\sG$ inducing lifted edge-weights~$\tilde{\sx}$ on~$\widetilde{\sG}$.
Let~$\widetilde{\G}$ (resp.~$\G$) be the graph associated with~$\widetilde{\sG}$ (resp.~$\sG$) as in Example~\ref{ex:Delta}.
Note that the graphs~$\widetilde{\G}$ and~$\G$ remain finite and connected, and the covering map~$\widetilde{\sG}\to\sG$ trivially extends to a covering
map~$\widetilde{\G}\to\G$. By Example~\ref{ex:Delta} and Corollary~\ref{cor:1}, we know
that~$\Delta_{\widetilde{\sG}}=\A_{\widetilde{\G}}$
is conjugate to~$\A_\G\oplus\A_\G^{\rho}=\Delta_\sG\oplus\Delta_\sG^{\rho}$ for some
representation~$\rho$ of~$\pi_1(\G,v_0)$. Therefore,
setting~$P_\sG^\rho(\lambda)\coloneqq \det(\lambda\operatorname{I}-\Delta^{\rho}_\sG)\in\C[\sx,\lambda]$,
we have the equality
\[
P_{\widetilde{\sG}}(\lambda)=P_{\sG}(\lambda)\cdot P^\rho_{\sG}(\lambda)\in\C[\sx,\lambda]\,.
\]
Observe that~$P_{\widetilde{\sG}}(\lambda)$ and~$P_{\sG}(\lambda)$ belong to~$\Z[\sx,\lambda]$, so~$P^\rho_{\sG}(\lambda)$ belongs
to the intersection of~$\C[\sx,\lambda]$ with the field of fractions~$Q(\Z[\sx,\lambda])=\Q(\sx,\lambda)$, i.e. it belongs to the ring~$\Q[\sx,\lambda]$.
Since the leading~$\lambda$-coefficient of~$P_{\sG}(\lambda)$ is equal to~$1$, the greatest common divisor of its coefficients is~$1$.
An application of Gauss's lemma (see e.g.~\cite[Chapter~IV.2, Corollary~2.2]{LangAlgebra}) now implies that~$P^\rho_{\sG}(\lambda)$ belongs to~$\Z[\sx,\lambda]$.
In conclusion,~$P_{\sG}(\lambda)$ divides~$P_{\widetilde{\sG}}(\lambda)$ in~$\Z[\sx,\lambda]$.

By the extended matrix-tree theorem stated above, we have~$P_{\sG}(-1)=\pm\sZ_{\mathit{RSF}}(\sG,\sx)$
divides~$P_{\widetilde{\sG}}(-1)=\pm\sZ_{\mathit{RSF}}(\widetilde{\sG},\tilde{\sx})$ in~$\Z[\sx]$,
proving the second claim.

To show the first one, consider again the
equation~$P_{\widetilde{\sG}}(\lambda)=P_{\sG}(\lambda)\cdot P^\rho_{\sG}(\lambda)$ in~$\Z[\sx,\lambda]$,
and observe that~$P_{\widetilde{\sG}}(\lambda)$ and~$P_{\sG}(\lambda)$ are both multiples of~$\lambda$.
Dividing both sides by~$\lambda$ and setting~$\lambda=~0$, the matrix-tree theorem
(in the form stated above) implies
\[
|\V(\widetilde{\sG})|\cdot\sZ_{\mathit{ST}}(\widetilde{\sG},\tilde{\sx})=\pm
|\V(\sG)|\cdot\sZ_{\mathit{ST}}(\sG,\sx)\cdot P^\rho_{\sG}(0)\,,
\]
i.e.~$\sZ_{\mathit{ST}}(\widetilde{\sG},\tilde{\sx})=\sZ_{\mathit{ST}}(\sG,\sx)\cdot g(\sx)$,
with~$g(\sx)=\frac{\pm 1}{\deg(\widetilde{\sG}/\sG)}P^\rho_{\sG}(0)\in\Q[\sx]$.
Since both~$\sZ_{\mathit{ST}}(\widetilde{\sG},\tilde{\sx})$ and~$\sZ_{\mathit{ST}}(\sG,\sx)$
belong to~$\Z[\sx]$ and the greatest common divisor of the coefficients
of~$\sZ_{\mathit{ST}}(\sG,\sx)$ is~$1$, one more application of Gauss's lemma yields that~$g(\sx)$
lies in~$\Z[\sx]$, and concludes the proof.
\end{proof}

\subsection{Perfect matchings}
\label{sub:dimers}

In this subsection, we review some applications of Theorem~\ref{thm:main} to perfect matchings,
and more generally to the dimer model.

Recall that a \new{perfect matching} (or \new{dimer configuration}) in a graph~$\G$ is a
family of edges~$M\subset\E$ such that
each vertex of~$\G$ is adjacent to a unique element of~$M$.
If~$\G$ is finite and endowed with symmetric edge-weights~$\sx=\{\sx_\se\}_{\se\in\E}$, then
one defines the \new{dimer partition function} of~$\G$ as
\[
\sZ_\text{dimer}(\G,\sx)=\sum_{M}\prod_{\se\in M}\sx_\se\,,
\]
the sum being over all perfect matchings in~$\G$.
Note that if all the weights are equal to~$1$, then~$\sZ_\text{dimer}(\G,1)$ simply counts the
number of perfect matchings in~$\G$.

Now, assume that~$\G$ is embedded in the plane, and endowed with an orientation of its edges so that
around each face of~$\G\subset\R^2$, there is an odd number of edges oriented clockwise.
Let~$x=\{x_e\}_{e\in\EE}$ be the anti-symmetric edge-weights obtained as in Example~\ref{ex:Kasteleyn},
and let~$\A_\G$ be the associated weighted skew-adjacency operator. By Kasteleyn's celebrated theorem~\cite{Ka1,Ka2},
the Pfaffian of~$\A_\G$ is equal to~$\pm\sZ_\text{dimer}(\G,\sx)$.

With this powerful method in hand, we can try to use Theorem~\ref{thm:main} in studying the dimer
model on symmetric graphs. Quite unsurprisingly, the straightforward applications of our theory are
not new. Indeed, the only divisibility statement that we obtain via Corollary~\ref{cor:1}
is the following known result (see Theorem~3 of~\cite{Joc} for the bipartite case, and Section~IV.C
of~\cite{Kuperberg} for a general discussion).

\begin{prop}
\label{prop:dimer}
Fix a planar, finite, connected weighted graph~$(\widetilde{\G},\tilde{\sx})$ invariant
under rotation around a point in the complement of~$\widetilde{\G}$,
of angle~$\frac{2\pi}{d}$ for some odd integer~$d$.
Let~$(\G,\sx)$ be the resulting quotient weighted graph.
Then, the partition function~$\sZ_\mathrm{dimer}(\G,\sx)$
divides~$\sZ_\mathrm{dimer}(\widetilde{\G},\tilde{\sx})$ in the ring~$\Z[\sx]$.
\end{prop}

\begin{proof}
Let us fix an orientation of the edges of~$\G$ satisfying the clockwise-odd condition.
It lifts to an orientation of~$\widetilde{\G}$ which trivially satisfies the same condition
around all faces except possibly the face containing the center of rotation; for this latter face,
it does satisfy the condition since~$d$ is odd.
Hence, we have a~$d$-fold cyclic covering of connected weighted
graphs~$(\widetilde{\G},\tilde{x})\to(\G,x)$, and Corollary~\ref{cor:1} can be applied.
Together with Kasteleyn's theorem, it yields the following equality in~$\C[\sx]$:
\[
\sZ_\text{dimer}(\widetilde{\G},\tilde{\sx})^2=\det(\A_{\widetilde{\G}})=
\det(\A_\G)\det(\A_\G^\rho)=\sZ_\text{dimer}(\G,\sx)^2\cdot\det(\A_\G^\rho)\,.
\] 
This ring being factorial, it
follows that~$\sZ_\text{dimer}(\widetilde{\G},\tilde{\sx})=\sZ_\text{dimer}(\G,\sx)\cdot g$
for some~$g\in\C[\sx]$. The fact that~$g$ belongs to~$\Z[\sx]$ follows from Gauss's lemma
as in the proof of Theorem~\ref{thm:tree}.
\end{proof}

Our approach is limited by the fact that we consider graph coverings~$\widetilde{\G}\to\G$ which,
in the case of normal coverings, correspond to free actions of~$G(\widetilde{\G}/\G)$ on~$\widetilde{\G}$.
For this specific question of enumerating dimers on symmetric planar graphs, the discussion of Section~IV
of~\cite{Kuperberg} is more complete, as non-free actions are also considered.

\medskip

However, our approach is quite powerful when applied to non-planar graphs.
Indeed, recall that Kasteleyn's theorem can be extended to weighted graphs embedded in
a closed (possibly non-orientable) surface~$\Sigma$, but the computation of the
dimer partition function requires the Pfaffians of~$2^{2-\chi(\Sigma)}$ different
(possibly complex-valued) skew-adjacency matrices~\cite{Tes,C-R,Cim}.
In particular, the partition function of any graph embedded in the torus~$\mathbb{T}^2$ is given by~$4$ Pfaffians.
For the Klein bottle~$\K$, we also need~$4$ Pfaffians, which turn out to be two pairs of conjugate complex numbers, so~$2$ well-chosen Pfaffians are sufficient. We now illustrate the use of Theorem~\ref{thm:main} in these two cases.

Let us first consider a toric graph~$\G\subset\mathbb{T}^2$, and let~$\widetilde{\G}=\G_{mn}$
denote the lift of~$\G$ by the natural~$m\times n$ covering of the torus by itself.
This covering is normal with Galois group~$G(\widetilde{\G}/\G)\simeq\Z/m\Z\oplus\Z/n\Z$.
This group being abelian, all the irreducible representations are of degree~$1$;
more precisely, they are given by~$\{\rho(z,w)\}_{w^m=1,\,z^n=1}$, where~$\rho(z,w)$
maps a fixed generator of~$\Z/m\Z$ (resp.~$\Z/n\Z$) to~$w\in\C^*$ (resp.~$z\in\C^*$).
Writing~$P_{mn}(z,w)=\det(\A_{\G_{mn}}^{\rho(z,w)})$ and~$P_{1,1}=P$, Corollary~\ref{cor:2} immediately yields
the equality
\[
P_{mn}(1,1)=\prod_{z^n=1}\prod_{w^m=1}P(z,w)\,.
\]
This is the well-known Theorem~3.3 of~\cite{KOS}, a result of fundamental importance in the study of the dimer
model on biperiodic graphs.

\begin{figure}
\includegraphics[width=9cm]{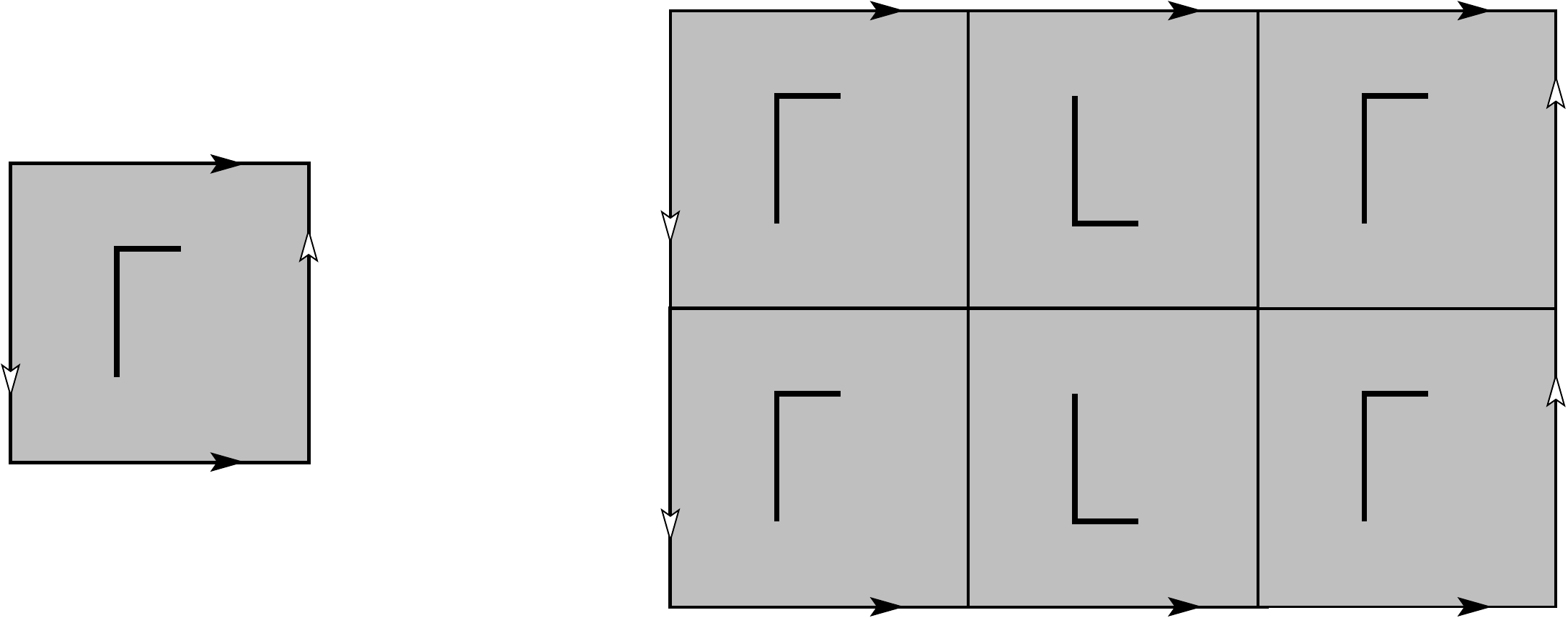}
\caption{A graph~$\G$ embedded in the Klein bottle~$\K$ (pictured as a square with opposite sides identified
according to the arrows), and the lift~$\G_{mn}\subset\K$, here with~$m=2$ and~$n=3$.}
\label{fig:Gmn}
\end{figure}

Let us now consider a weighted graph~$\G$ embedded in the Klein bottle~$\K$,
and let~$\widetilde{\G}=\G_{mn}$ denote the lift of~$\G$ by the natural~$m\times n$
cover~$\K_{mn}\to\K$ of the Klein bottle by itself (with~$n$ odd), as illustrated in Figure~\ref{fig:Gmn}.
Now, we can interpret the two skew-adjacency matrices of~$\widetilde{\G}=\G_{mn}$
used in the computation of the corresponding dimer partition function as
weighted adjacency operators twisted by~$1$-dimensional representations~$\rho,\rho'$ of~$\pi_1(\K_{mn})<\pi_1(\K)$.
Using Theorem~\ref{thm:main}, we see that these matrices are conjugate to
the skew-adjacency operators on~$\G\subset\K$ twisted by the corresponding
induced representations~$\rho^\#,(\rho')^\#$ of~$\pi_1(\K)$.
Unlike that of the torus, the fundamental group of the Klein bottle is not abelian,
so the representations~$\rho^\#,(\rho')^\#$ need not split as products of~$1$-dimensional
representations.
It turns out that they split as products of representations of degree~$1$ and~$2$,
yielding a closed formula for~$\sZ_\text{dimer}(\G_{mn},x)$ in terms of determinants of~$\A_\G^\tau$,
with~$\tau$ of degree~$1$ and~$2$.
This result is at the core of the study of the dimer model on Klein bottles of the first-named
author~\cite{Cimasoni-Klein}.

\medskip

As a final remark, let us note that all the considerations of this subsection
can be applied equally well to the Ising model,
either via the use of Kac-Ward matrices~\cite{K-W}, or
via skew-adjacency matrices on the associated Fisher graph~\cite{Fisher}.

\subsection{Further combinatorial applications}
\label{sub:further}

We conclude this article with a very brief and informal description of additional applications of
our results.

\medskip

As discovered by Forman~\cite{Forman}, the determinant of~$\Delta_\sG^\rho$ with~$\deg(\rho)=1$ can be
expressed as a sum over \new{cycle-rooted spanning forests} (CRSFs) in~$\sG$, each forest being counted with a
complex weight depending on~$\rho$.
If there is a finite connected covering~$\widetilde{\sG}\to\sG$ and a degree~$1$ representation
of~$\pi_1(\widetilde{\sG},\tilde{v}_0)$ such that the induced representation of~$\pi_1(\sG,v_0)$
admits a degree~$1$ subrepresentation~$\rho'$, then the CRSF partition function on~$\sG$ twisted 
by~$\rho'$ divides the partition function on~$\widetilde{\sG}$ twisted by~$\rho$, in the ring~$\C[\sx]$.
Furthermore, in the case of a normal abelian covering of degree~$d$, Corollary~\ref{cor:2} gives a factorisation
of the CRSF partition function of~$\widetilde{\sG}$ in terms of~$d$ CRSF partition functions of~$\sG$.

\medskip

Finally, let~$X$ be a finite~CW-complex
of dimension~$r$ with weights~$\sx=(\sx_\se)_\se$ associated
to the cells of top dimension. Let~$\sG$ be the weighted graph with vertex set given by the~$(r-1)$-dimensional cells
of~$X$, two such vertices being connected by an unoriented edge of~$\sG$ each time they are in the boundary
of an~$r$-dimensional cell. (Note that if~$r=1$, then the~$1$-dimensional cell complex~$X$ is nothing but the
geometric realisation of the graph~$\sG$.) Finally, let~$\G$ denote the weighted graph obtained from~$\sG$ as in
Example~\ref{ex:Delta}. Then, the resulting operator~$\A_\G$ is the \new{Laplacian}~$\Delta_X$ acting on~$r$-cells of~$X$.
This operator can be used to count so-called \new{higher dimensional rooted forests}
in~$X$, see~\cite{Kalai,B-K} and references therein.
Using Corollary~\ref{cor:1}, it is now straightforward to prove that, given any finite cover~$\widetilde{X}\to X$,
the corresponding rooted forest partition function of~$X$ divides the rooted forest partition function
of~$\widetilde{X}$, extending Theorem~\ref{thm:tree} to higher dimensional objects.

\bigskip

\noindent
\textbf{Acknowledgements.} D.C. thanks Pierre de la Harpe and Anders Karlsson for
useful conversations. The authors thank the referees for their constructive comments,
and Thierry L{\'e}vy for suggesting a simpler proof of the main result.

\nocite{*}
\bibliographystyle{amsplain-ac}
\bibliography{coverings-biblio-v2}

\appendix

\section{Addendum}

After our paper was in press, we became aware of the article~\cite{CDHGZZ} whose main result is very similar to Theorem~\ref{thm:tree}.
In a nusthell, Theorem~1.3 of~\cite{CDHGZZ} extends the spanning tree part of our Theorem~\ref{thm:tree}
from graphs to digraphs, showing the divisibility of the partition functions enumerating the corresponding
combinatorial objects, so-called {\em rooted arborescences\/}.
Note that this result can easily be obtained by our methods using the digraph version of the matrix-tree theorem.

In Conjecture~1.7 of~\cite{CDHGZZ}, it is asked whether the (integer) coefficients of the quotient polynomial are always non-negative. 
Moreover, Conjecture~5.5 of~\cite{CDHGZZ} asserts that this polynomial can be expressed as a sum over tuples of vector fields on~$\sG$.
The authors prove their conjectures in the case of coverings of degree $2$.
The aim of this addendum is to provide an affirmative answer to these conjectures in the case of arbitrary degree coverings of (undirected) graphs.

To state the corresponding result, recall that a \new{(non-zero) vector field} on a directed graph~$\sG$ is a directed subgraph of~$\sG$ consisting of one outgoing edge for each vertex of~$\sG$. 
We let~$\mathcal{V}(\sG)$ denote the set of vector fields on~$\sG$.
Note that vector fields appear in \cite{Forman} and are also known as {\em oriented cycle rooted spanning forests\/}, see e.g.~\cite{Kenyon}.

\begin{theorem}
\label{thm:addendum}
Let~$\widetilde{\sG}$ be a covering graph of degree~$(d+1)$ of a finite connected graph~$\sG$ endowed with
edge-weights~$\sx=\{\sx_\se\}_{\se\in\E}$, and let~$\tilde{\sx}$ denote these weights lifted to
the edges of~$\widetilde{\sG}$. Then~$\sZ_{\mathit{ST}}(\sG,\sx)$
divides~$\sZ_{\mathit{ST}}(\widetilde{\sG},\tilde{\sx})$ in the ring~$\Z[\sx]$, and the quotient can be expressed as
\begin{equation}\label{eq:vector-fields}
\frac{\sZ_{\mathit{ST}}(\widetilde{\sG},\tilde{\sx})}{\sZ_{\mathit{ST}}(\sG,\sx)}\;\;\;=\sum_{(\gamma_1,\dots,\gamma_{d})\in\mathcal{V}(\sG)^d} f(\gamma_1,\dots,\gamma_{d})
\prod_{k=1}^d \prod_{e\in \gamma_k} \sx_e
\end{equation}
for some non-negative integer-valued map~$f\colon\mathcal{V}(\sG)^d\to\{0,1,2,\dots\}$.
\end{theorem}

Let us stress that the right-hand side of \eqref{eq:vector-fields} is not written as a linear combination of distinct monomials: in particular, upon shuffling elements of a $d$-tuple of distinct vector fields, one obtains the same monomial $\prod_{k=1}^d \prod_{e\in \gamma_k} \sx_e$.

\begin{proof}
Let us use the same notations as in the proof of Theorem~\ref{thm:tree}. We know from that theorem that the quotient $\sZ_{\mathit{ST}}(\widetilde{\sG},\tilde{\sx})/\sZ_{\mathit{ST}}(\sG,\sx)$ is a homogeneous integer-coefficient polynomial equal to $\det \Delta_{\sG}^\rho\in\Z[\sx]$. We thus need to show the following two facts:
\begin{itemize}
\item[1.] the coefficients of this polynomial are non-negative;
\item[2.] for each monomial $\prod_{\sf e\in \E} \sx_\se^{n_\se}$, there exists $(\gamma_1,\ldots, \gamma_d)\in \mathcal{V}(\sG)^d$ such that for each~$\se \in \E$, whose oriented versions we denote by $\{e,\overline{e}\}$, we have $n_\se=\sum_{k=1}^d \left(\mathbbm{1}_{\{e\in \gamma_k\}}+\mathbbm{1}_{\{\overline{e}\in \gamma_k\}}\right)$.
\end{itemize}

\smallskip

We start with the first assertion concerning the non-negativity of the coefficients.
Let us first recall how the representation~$\rho$ of~$\pi_1(\sG,v_0)$ is obtained (see Example~\ref{ex:ind}), and investigate some of its properties. 
Consider the representation~$1^\#$ of~$\pi_1(\sG,v)$ induced by the trivial
representation of~$\pi_1(\widetilde{\sG},\tilde{v})$: it is given by the action by left multiplication of~$\pi_1(\sG,v)$
on the vector space~$Z$ with basis~$\pi_1(\sG,v)/\pi_1(\widetilde{\sG},\tilde{v})$.
Fix the inner product~$\langle\cdot, \cdot\rangle$ on~$Z$ with respect to which this basis is orthonormal.
Since~$1^\#$ acts by permutation on the elements of this basis, this representation is unitary with respect to
the inner product defined above.
Moreover, since~$Z$ is finite-dimensional, we have an orthogonal decomposition~$Z=\C\oplus W$,
where~$\C$ is generated by the sum of the basis vectors and~$W$ consists of vectors with vanishing sum of coordinates.
This leads to a decomposition~$1^\#=1\oplus\rho$,
with~$\rho$ the representation appearing in the polynomial~$\det\Delta_{\sG}^\rho\in\Z[\sx]$.
Since~$1^\#$ is unitary, 
we have now shown that~$\rho$ is a
unitary representation.
Hence, we are left with the proof that the coefficients of the polynomial~$\det \Delta_{\sG}^\rho$ are non-negative whenever~$\rho$ is a unitary representation.

To do so, let us consider a unitary connection~$\Phi=(\varphi_e)_{e\in\EE}$ representing
the unitary representation~$\rho\in \GL(W)$ (recall subsection~\ref{sub:bundle}),
which defines a twisted Laplacian $\Delta_\sG^\Phi$ whose gauge-equivalence class is~$\Delta_{\sG}^\rho$.
In order to factor~$\Delta_\sG^\Phi$, we introduce a refinement of the connection as in~\cite{Kenyon}:
for each directed edge~$e\in\EE$, consider unitary automorphisms~$\varphi_{s(e),e},\varphi_{e,t(e)}\in \GL(W)$ such that~$\varphi_{\overline{e},t(\overline{e})}=\varphi_{s(e),e}^{-1}$ and~$\varphi_e=\varphi_{s(e),e}\circ \varphi_{e,t(e)}$. 
Let us further define the spaces of~$W$-valued forms on~$\sG$ via 
\[
\Omega^0(\sG,W)=\{f:\V\to W\}=W^\V\quad\text{ and }\quad\Omega^1(\sG,W)=\{\alpha:\EE \to W \;|\; \alpha_e=-\alpha_{\overline{e}}\}\simeq W^\E\,.
\]
Note that the inner product~$\langle\cdot, \cdot\rangle$ on~$W$ naturally extends to inner products on~$\Omega^0(\sG,W)$ and~$\Omega^1(\sG,W)$ via
\[
(f,g)_{\Omega^0}=\sum_{v\in \V} \langle f(v),g(v)\rangle \quad \text{and}
\quad (\alpha,\beta)_{\Omega^1}=\frac{1}{2}\sum_{e\in \EE} \langle \alpha(e), \beta(e)\rangle
\]
for~$f,g\in \Omega^0(\sG,W)$ and~$\alpha,\beta\in \Omega^1(\sG,W)$.
Next, we define twisted coboundary and boundary maps
\[
\delta_\Phi\colon \Omega^0(\sG,W)\to\Omega^1(\sG,W)\quad\text{and}\quad\partial_\Phi\colon \Omega^1(\sG,W)\to\Omega^0(\sG,W)
\]
by
\[
\delta_\Phi f(e)=\varphi_{e,t(e)} f\left(t(e)\right)-\varphi_{s(e),e}^{-1} f\left(s(e)\right)  \quad \text{and} \quad \partial_\Phi \alpha(v)=-\sum_{e\in \EE_v} \varphi_{s(e),e} \,\alpha(e)
\]
for~$f\in \Omega^0(\sG,W),\alpha\in \Omega^1(\sG,W),v\in \V$ and~$e\in \EE$.
Since the automorphisms~$\varphi_{s(e),e}$ and~$\varphi_{e,t(e)}$ are unitary, one can now check that
the maps~$\delta_\Phi$ and $\partial_\Phi$ are adjoint of one another. In other words,
for all~$f\in \Omega^0(\sG,W)$ and~$\alpha\in \Omega^1(\sG,W)$, there is an equality between
\[
(f,\partial_\Phi\alpha)_{\Omega^0}=\sum_{v\in \V} \langle f(v),\partial_\Phi\alpha(v)\rangle
=-\sum_{v\in \V}\sum_{e\in\EE_v}\langle f(v),\varphi_{s(e),e}\alpha(e)\rangle=-\sum_{e\in\EE}\langle f(s(e)),\varphi_{s(e),e}\alpha(e)\rangle
\]
and
\begin{align*}
(\delta_\Phi f,\alpha)_{\Omega^1}&=\frac{1}{2}\sum_{e\in\EE}\langle\delta_\Phi f(e),\alpha(e)\rangle
=\frac{1}{2}\sum_{e\in\EE}\left(\langle\varphi_{e,t(e)}f(t(e)),\alpha(e)\rangle-\langle\varphi_{s(e),e}^{-1}f(s(e)),\alpha(e)\rangle\right)\\
	&=\frac{1}{2}\sum_{e\in\EE}\left(-\langle f(s(\overline{e})),\varphi_{s(\overline{e}),\overline{e}}\alpha(\overline{e})\rangle-\langle f(s(e)),\varphi_{s(e),e}\alpha(e)\rangle\right)\\
	&=-\sum_{e\in\EE}\langle f(s(e)),\varphi_{s(e),e}\alpha(e)\rangle\,.
\end{align*}

Finally, we let~$\diag_{\sf x}$ denote the block-scalar map on~$\Omega^1(\sG,W)$ equal to~${\sx}_e$ on the block~$W$ indexed by the edge~$e$. For each~$f\in \Omega^0(\sG,W)$ and~$v\in \V$, we now have
\begin{align*}
(\partial_\Phi \circ \diag_{\sx} \circ \delta_\Phi)(f)(v)&=\sum_{e\in \EE_v}\sx_e\varphi_{s(e),e}\left(\varphi_{s(e),e}^{-1} f\left(s(e)\right)- \varphi_{e,t(e)} f\left(t(e)\right)\right)\\
	&=\sum_{e\in \EE_v} {\sx}_{e} \left[f\left(s(e)\right)-\varphi_ef\left(t(e)\right)\right]=\Delta_{\sG}^\Phi(f)(v)\,.
\end{align*}
In view of this factorisation~$\Delta_{\sG}^\Phi=\partial_\Phi \circ \diag_{\sx} \circ \delta_\Phi$,
we can now compute $\det \Delta_\sG^\rho=\det \Delta_\sG^\Phi$ by an application of the Cauchy--Binet formula. For this purpose, we identify linear maps with matrices in orthonormal bases of~$\Omega^0(\sG,W)$ and~$\Omega^1(\sG,W)$ induced by concatenating copies of a given orthonormal basis of~$W$. 

For a matrix~$M$, we will denote by~$M^I_J$ the sub-matrix of~$M$ indexed by rows~$I$ and columns~$J$, and we omit~$I$ or~$J$ in this notation when it is equal to the full set of indices. 
With a slight abuse of notation, we denote by $\delta$ the matrix of $\delta_\Phi$ and by $\partial$ the matrix of $\partial_\Phi$ with respect to the above chosen orthonormal bases. 
Setting~$d=\dim W$,
the Cauchy--Binet formula yields
\[
\det \Delta_\sG^\Phi=\sum_{\substack{J\subset \E\times \{1,\ldots, d\}\\ \vert J\vert=d \vert \V\vert}} \det \partial_J\,\det (\diag_{\sf x})_J ^J\,\det \delta^J \,.
\]
Since~$\det (\diag_{\sf x})^J_J=\prod_{\se\in\E}\sx_\se^{n_\se(J)}$ with~$n_\se(J)=\vert J\cap\{\se\}\times \{1,\ldots, d\}\vert$, we find that  the coefficient of an arbitrary monomial~$\prod_{\se\in\E}\sx_{\se}^{n_\se}$
in the above expansion is equal to 
\begin{equation}\label{eq:cauchy-binet}
\sum_{\substack{J\subset\E\times\{1,\ldots, d\},\;\vert J\vert=d \vert \V\vert\\ n_\se(J)=n_\se, \forall \se\in \E}} \\ \det \partial_J\,\det\delta^J
\;=\sum_{\substack{J\subset\E\times\{1,\ldots, d\},\;\vert J\vert=d \vert \V\vert\\ n_\se(J)=n_\se, \forall \se\in \E}} \left\vert \det \partial_J \right\vert^2\,,
\end{equation}
using the fact that the maps~$\delta_\Phi$ and~$\partial_\Phi$ are adjoint.
This shows that this coefficient is non-negative, and concludes the proof of the first assertion.

\smallskip

Let us now prove the second one. Let $(n_\se)_{\se\in \E}$ be a tuple of non-negative integers which corresponds to one of the monomials of our polynomial~$\det \Delta_\sG^\Phi$. Since the coefficient of this monomial is given by the right-hand side of \eqref{eq:cauchy-binet} and is non-zero, there exists $J\subset\E\times\{1,\ldots, d\}$ satisfying $\vert J\vert=d \vert \V\vert$ and $n_\se(J)=n_\se, \forall \se\in \E$, such that
\begin{equation}\label{eq:one-term}
\left\vert \det \partial_J \right\vert^2\ne 0\,.
\end{equation}
Let us write $\partial=(\partial^{i}_{j})_{i\in \V\times\{1,\ldots, d\},j\in \E\times\{1,\ldots, d\}}$. By~\eqref{eq:one-term}, and using the classical expansion of the determinant as a sum over permutations, there exists a bijection $\sigma:\V\times\{1,\ldots, d\}\to J$ such that 
\begin{equation}\label{eq:prod-non-zero}
\prod_{i\in \V\times\{1,\ldots, d\}} \partial^i_{\sigma(i)}\ne 0\,.
\end{equation}
Fix~$k\in \{1,\ldots, d\}$. For each $v\in \V$, set $i=(v,k)$ and let $\se\in \E$ and $\ell\in \{1,\ldots, d\}$ be defined by $\sigma(i)=(\se,\ell)$. Since $\partial^i_{\sigma(i)}\ne 0$ by \eqref{eq:prod-non-zero}, we see that~$v$ must be one of the boundary-vertices of~$\se$. Let $e\in \EE_v$ be the oriented version of~$\se$ for which $v=s(e)$ and set~$\gamma_k(v)=e$. (If~$\se$ is a self-loop, simply pick an arbitrary orientation.)
By definition, the collection $\left(\gamma_k(v)\right)_{v\in \V}\in \prod_{v\in \V} \EE_v$ defines an element $\gamma_k$ of~$\mathcal{V}(\sG)$. Furthermore, by construction and since $\sigma$ is a bijection, for each edge $\se\in \E$, whose oriented versions we denote by $\{e,\overline{e}\}$, we have
\begin{align*}
n_\se&= n_\se(J)=\Big\vert \sigma^{-1}\left(\{\se\}\times\{1,\ldots, d\}\right)\Big\vert\\
&=\sum_{k=1}^d \sum_{v\in \V} \mathbbm{1}_{\{\sigma(v,k)\in\{\se\}\times\{1,\ldots, d\}\}}
=\sum_{k=1}^d \left(\mathbbm{1}_{\{e\in\gamma_k\}}+\mathbbm{1}_{\{\overline{e}\in \gamma_k\}}\right)\,.
\end{align*}
This concludes the proof. 
\end{proof}

We note that the assertion proved in the second part of the proof of Theorem~\ref{thm:addendum} can also be obtained as an easy corollary of~\cite[Theorems 5.1]{KL2}. A slightly stronger version of that assertion appears in~\cite{KL6} as well.

As a final remark, note that our method of proof adapts partly to the setup of digraphs. However, it fails, without further input, to prove the full Conjecture~5.5 of~\cite{CDHGZZ}. Indeed, our argument to show the non-negativity of~$f$ relies on a symmetry property which does not hold for general digraphs.

\end{document}